\documentclass[12pt,twoside,final,amsfonts]{amsart}
\usepackage[draft=false]{hyperref}
\usepackage{amssymb}
\usepackage[dvips,final]{graphicx,psfrag}
\usepackage{times,a4wide}

\usepackage[alphabetic,backrefs]{amsrefs}

\theoremstyle{plain}
\newtheorem*{theorem*}{Theorem}

\newtheorem{theorem}{Theorem}

\newtheorem*{proposition*}{Proposition}
\newtheorem{corollary}[theorem]{Corollary}
\newtheorem*{corollary*}{Corollary}
\newtheorem{lemma}[theorem]{Lemma}
\newtheorem*{lemma*}{Lemma}

\theoremstyle{definition}

\newtheorem*{remark*}{Remark}
\newtheorem*{remarks*}{Remarks}
\newtheorem*{conjecture*}{Conjecture}
\theoremstyle{definition}

\newcommand{\C}{\mathbb{C}}
\newcommand{\B}{\mathbb{B}}

\newcommand{\E}{\mathbb{E}}

\newcommand{\eps}{\epsilon}

\newcommand{\Aut}{\operatorname{Aut}}

\newcommand{\Cov}{\operatorname{Cov}}

\newcommand{\Var}{\operatorname{Var}}

\title[Gaussian Analytic functions in the unit ball]
{Gaussian Analytic functions in the unit ball}

\author[Jeremiah Buckley] {Jeremiah Buckley}
\address{School of Mathematical Sciences, Tel Aviv University, Tel Aviv 69978, Israel}
\email{buckley@post.tau.ac.il}

\author[Xavier Massaneda] {Xavier Massaneda}

\address{Departament de Matem\`atica Aplicada i An\`alisi,
Universitat  de Bar\-ce\-lo\-na, Gran Via 585, 08071-Bar\-ce\-lo\-na, Spain}
\email{xavier.massaneda@ub.edu}

\author[Bharti Pridhnani] {Bharti Pridhnani}

\address{Departament de Matem\`atica Aplicada i An\`alisi,
Universitat  de Bar\-ce\-lo\-na, Gran Via 585, 08071-Bar\-ce\-lo\-na, Spain}
\email{bharti.pridhnani@ub.edu}

\date{\today}

\keywords{}

\subjclass{}

\begin{document}

\begin{abstract} 
We study some properties of hyperbolic Gaussian analytic functions of intensity $L$ in the unit ball of $\C^n$. First we deal with the asymptotics  of fluctuations of linear statistics as $L\to\infty$. Then we estimate the probability of large deviations (with respect to the expected value) of such linear statistics and use this estimate to prove a hole theorem.
\end{abstract}

\maketitle

\section*{Introduction}

Let $\mathbb B_n$ denote the unit ball in $\mathbb C^n$ and let $\nu$ denote the Lebesgue measure in $\mathbb C^n$ normalised so that $\nu(\B_n)=1$. Explicitly $\nu=\frac{n!}{\pi^n}dm=\beta^n$, where $dm$ is the Lebesgue measure and $\beta=\frac {i}{2\pi}\partial\bar\partial |z|^2$ is the fundamental form of the Euclidean metric.

For $L>n$ consider the weighted Bergman space
\[
 B_L(\mathbb B_n)=\bigl\{f\in H(\B_n) : \|f\|_{n,L}^2:=c_{n,L} \int_{\B_n} |f(z)|^2 (1-|z|^2)^{L}  d\mu(z)< +\infty\bigr\} ,
\]
where  
\begin{equation}\label{m-invariant}
 d\mu(z)=\frac {d\nu(z)}{(1-|z|^2)^{n+1}},
\end{equation}
and $c_{n,L}=\frac{\Gamma(L)}{n!\Gamma(L-n)}$ is chosen so that $\|1\|_{n,L}=1$. 

Let
\[
 e_{\alpha}(z)=\left(\frac{\Gamma(L+|\alpha|)}{\alpha!\Gamma(L)}\right)^{1/2} z^\alpha
\]
denote the normalisation of the monomial $z^\alpha$ in the norm $\|\cdot\|_{n,L}$, so that $\{e_\alpha\}_\alpha$ is an orthonormal basis of $B_L(\mathbb B_n)$. As usual, here we denote $z=(z_1,\dots,z_n)$ and use the multi-index notation $\alpha=(\alpha_1,\dots,\alpha_n)$, $\alpha !=\alpha_1!\cdots \alpha_n!$, $|\alpha|=|\alpha_1|+\cdots +|\alpha_n|$ and $z^\alpha=z_1^{\alpha_1}\cdots z_n^{\alpha_n}$.

The \emph{hyperbolic Gaussian analytic function} (GAF) of \emph{intensity} $L$ is defined as
\[
 f_L(z)=\sum_\alpha a_\alpha  \left(\frac{\Gamma(L+|\alpha|)}{\alpha!\Gamma(L)}\right)^{1/2} z^\alpha \qquad z\in\B_n,
\]
where $a_\alpha $ are i.i.d. complex Gaussians of mean 0 and variance 1 ($a_\alpha \sim N_{\C}(0,1)$).

We choose the orthonormal basis $\{e_{\alpha}\}_\alpha$ for convenience, but any other basis would produce the same covariance kernel (see below) and therefore the same results.

The sum defining $f_L$ can be analytically continued to $L>0$, which we assume henceforth.

The characteristics of the hyperbolic GAF are determined by its covariance kernel, which is given by (see \cite{ST1}*{Section 1}, \cite{Stoll}*{p.17-18})
\begin{align*}
 K_L(z,w)&=\E[f_L(z)\overline{f_L(w)}]= \sum_\alpha \frac{\Gamma(L+|\alpha|)}{\alpha!\Gamma(L)} z^\alpha \bar w^\alpha=
 \sum_{m=0}^\infty   \frac{\Gamma(L+m)}{\Gamma(L)} \sum_{\alpha : |\alpha|=m} \frac 1{\alpha!} z^\alpha \bar w^\alpha\\
 &=
 \sum_{m=0}^\infty \frac{\Gamma(L+m)}{m! \Gamma(L)} (z\cdot\bar w)^m=
 \frac 1{(1-z\cdot \bar w)^{L}}\ .
\end{align*}

A main feature of the hyperbolic GAF is that the distribution of its zero set 
\[
 Z_{f_L}=\left\{z\in \B_n;\ f_L(z)=0\right\} 
\]
is invariant under the group $\Aut(\B_n)$ of holomorphic automorphisms of the ball. Given $w\in\B_n$ there exists $\phi_w\in\Aut(\B_n)$ such that $\phi_w(w)=0$ and $\phi_w(0)=w$, and all automorphisms are essentially of this form: for all $\psi\in\Aut(B_n)$ there exist $w\in\B_n$ and $\mathcal U$ in the unitary group such that $\psi=\mathcal U \phi_w$ (see \cite{Ru}*{2.2.5}). Then the \emph{pseudo-hyperbolic distance} $\varrho$ in $\B_n$ is defined as
\[
 \varrho(z,w)=|\phi_w(z)|,\quad z,w\in\B_n\ ,
\]
and the corresponding pseudo-hyperbolic balls as 
\[
 E(w,r)=\{z\in\B_n : \varrho(z,w)<r\}, \qquad r<1\ .
\]
There is an immediate relation between the normalised covariance kernel and the pseudo-hyperbolic distance, given by the identity
\begin{equation}\label{dh}
 1-|\phi_w(z)|^2=\frac{(1-|z|^2) (1-|w|^2)}{|1-\bar z\cdot w|^2}\ .
\end{equation}
The transformations
\[
 T_w(f)(z)=\left(\frac{1-|w|^2}{(1-\bar w\cdot z)^2}\right)^{L/2} f(\phi_w(z))
\]
are isometries of $B_L(\B_n)$, hence the random zero sets $Z_{f_L}$ and $Z_{f_L\circ\phi_w}$ have the same distribution. More specifically, the
distribution of the (random) integration current  
\[
 [Z_{f_L}]=\frac {i}{2\pi}\partial\bar\partial\log |f_L|^2\ ,
\]
is invariant under automorphisms of the unit ball. 

The Edelman-Kostlan formula (see \cite{HKPV}*{Section 2.4} and \cite{Sod}*{Theorem 1}) gives the so-called \emph{first intensity} of the GAF:
\[
\E[Z_{f_L}] =\frac{i}{2\pi}\partial \overline{\partial}\log K_L(z,z)=L\, \omega(z)\ ,
\]
where $\omega$ is the invariant form
\[
\omega(z)=\frac{i}{2\pi}\partial \overline{\partial}\log\bigl(\frac 1{1-|z|^2}\bigr)=\frac{1}{(1-|z|^2)^2}\frac{i}{2\pi}\sum^{n}_{j,k=1}[(1-|z|^2)\delta_{j,k}+z_k\overline{z_j}]dz_j\wedge d\overline{z_k}\ .
\]

Notice that $\mu=\omega^n$ is also invariant by $\Aut(\B_n)$ \cite{Stoll}*{p.19}.

In this paper we study some statistical properties of the zero variety $Z_{f_L}$ for large values of the intensity $L$. The outline of the paper is as follows.

In Section 1 we study the fluctuations of linear statistics as the intensity $L$ tends to $\infty$. Let $\mathcal D_{(n-1,n-1)}$ denote the space of compactly supported smooth forms of bidegree $(n-1,n-1)$. For  $\varphi\in\mathcal D_{(n-1,n-1)}$, consider the integral of $\varphi$ over $Z_{f_L}$:
\[
I_L(\varphi)=\int_{Z_{f_L}}\varphi=\int_{\B_n}  \varphi  \wedge [Z_{f_L}]  .
\]
By the Edelman-Kostlan formula,
\begin{equation}\label{E}
\E[I_{L}(\varphi)]=L\int_{\B_n}\varphi\wedge \omega\ .
\end{equation}
We compute the leading term in the limit as $L\to\infty$ of $\Var [I_L(\varphi)]$ and see that the rate of self-averaging of the integral of $I_L(\varphi)$ increases with the dimension. A quantitative statement is the following.

\begin{theorem}\label{LS}
 Let $\varphi\in \mathcal D_{(n-1,n-1)}$  and let $D\varphi$ be the function defined by $\frac i{2\pi}\partial\bar\partial\varphi=D\varphi d\mu$. Then
 \[
  \Var [I_L(\varphi)]=n! \zeta(n+2)\left(\int_{\B_n} (D\varphi)^2 d\mu\right)\frac 1{L^n}+\textrm{O}\left(\frac {\log L}{L^{n+1}}\right).
 \]
\end{theorem}
Notice that this shows a strong form of self-averaging of the volume $I_L(\varphi)$, in the sense that
\[
 \frac{\Var I_L(\varphi)}{(\E[I_L(\varphi)])^2}=\textrm O\left(\frac 1{L^{n+2}}\right).
\]
Notice also that the self-averaging increases with the dimension.

The same computations involved in the proof of the this theorem show the asymptotic normality of $I_L(\varphi)$, i.e., that the distributions of 
\[
\frac{I_L(\varphi)-\E[I_L(\varphi)]}{\sqrt{\Var [I_L(\varphi)]}}
\]
converge weakly to the (real) standard gaussian (Corollary~\ref{normality}).

The proofs are rather straight-forward generalisations of the proof for the one-dimensional case given by Sodin and Tsirelson \cite{ST1}, or the analogous result in the context of compact manifolds given by Shiffman and Zelditch.

Let $p_N$ be a Gaussian holomorphic polynomial in $\C\mathbb P^n$ or, more generally, a section of a power $L^N$ of a positive Hermitian line bundle $L$ over an $n$-dimensional K\"ahler manifold $M$. Given a test form $\varphi$  of bidegree $(n-1,n-1)$, define
\[
 I_{N}(\varphi)=\int_{Z_{p_N}}\varphi=\int_M\varphi\wedge [Z_{p_N}]\ .
\]
According to \cite{SZ10}*{Theorem 1}, as $N\to\infty$,
\[
 \Var [I_{N}(\varphi)]=\frac{\pi^{n-2}}{4} \zeta(n+2)  \|\partial\bar\partial \varphi\|^2_{L^2}\, \frac 1{N^n} +\textrm{O}(\frac 1{N^{n+1/2-\epsilon}})\ .
\]
The proof of this result is based on a bi-potential expression of $\Var [I_{N}(\varphi)]$ (see \eqref{bipotential}) together with good estimates of the covariance kernel, something we certainly have for the GAF in the ball.

In Section 2, we deal with large deviations. We study the probability that the deviation of $I_L(\varphi)$ from its expected value is at least a fixed proportion of $\E[I_L(\varphi)]$.

\begin{theorem}\label{smoothlargedeviations}
For all $\varphi\in \mathcal D_{(n-1,n-1)}$ and $\delta>0$, there exist $c>0$ and $L_0(\varphi,\delta,n)$ such that for all $L\geq L_0$,
\[
\mathbb P\left[\left\{\omega:\ |I_L(\varphi)-\E(I_L(\varphi))|>\delta \E(I_L(\varphi))\right\}\right]\leq e^{-cL^{n+1}}.
\]
\end{theorem}

Replacing $\delta\int_{\B_n}\varphi\wedge \omega$ by $\delta$ we get the equivalent formulation:
\[
\mathbb P\left[\left|\frac{1}{L}I_L(\varphi)-\int_{\B_n}\varphi\wedge\ \omega\right|>\delta \right]\leq e^{-cL^{n+1}}.
\]
Following the scheme of \cite{SZZ}*{pag.1994} we deduce a corollary that implies the upper bound in the hole theorem (Theorem~\ref{holethm} below). For a compactly supported function $\psi$ in $\B_n$ denote 
\[
 I_L(\psi)=\int_{Z_{f_L}}\psi\omega^{n-1} =\int_{\mathbb B_n} \psi \wedge \omega^{n-1}\wedge [Z_{f_L}]\ .
\]
Notice that \eqref{E} gives here
\[
 \E[I_L(\psi)]=L\int_{\B_n} \psi\, d\mu\ .
\]

In particular, and for an open set $U$ in the ball let $\chi_U$ denote its characteristic function and let  $I_L(U)=I_L(\chi_U)$. Then $\E[I_L(U)]=L\mu(U)$.

\begin{corollary}\label{conseqSLD}
Suppose that $U$ is an open set contained in a compact subset of $\B_n$. For all $\delta>0$ there exist $c>0$ and $L_0$ such that for all $L\geq L_0$,
\[
\mathbb P\left[\left|\frac{1}{L}I_L(U)-\mu(U)\right|>\delta\right]\leq e^{-cL^{n+1}}.
\]
\end{corollary}

The case $n=1$ of Theorem~\ref{smoothlargedeviations} is given in \cite{Bu-T13}*{Theorem 5.7}. Our proof is inspired by the methods of B. Shiffman, S. Zelditch and S. Zrebiec for the study of the analogous problem for compact K\"ahler manifolds. According to \cite{SZZ}*{Theorem 1.5}, given $\delta>0$, and letting $\omega$ denote the K\"ahler form of the manifold, 
\[
\mathbb P\left[ \left| \frac 1N \int_{Z_{p_N}}\varphi-\frac 1{\pi}\int_M  \omega\wedge \varphi\right| > \delta\right]\leq e^{-c N^{n+1}},
\]
where here $N$ indicates here the power of the positive Hermitian bundle over $M$.

In the last Section we study the probability that $Z_{f_L}$ has a pseudohyperbolic hole of radius $r$. By the invariance by automorphisms of the distribution of the zero variety, this is the same as studying the probability that $Z_{f_L}\cap B(0,r)=\emptyset$. 

\begin{theorem}\label{holethm}
Let $r\in (0,1)$ be fixed. There exist $C_1=C_1(n,r)>0$, $C_2=C_2(n,r)>0$ and $L_0$ such that for all 
$L\geq L_0$,
\[
e^{-C_{1}L^{n+1}} \leq \mathbb P\left[Z_{f_L}\cap B(0,r)=\emptyset\right]\leq e^{-C_{2}L^{n+1}}.
\]
\end{theorem}

This result is inspired again by an analogue for entire functions in the plane given by Sodin and Tsirelson \cite{ST3}. Let
\[
 \mathcal F_L=\{f\in H(\C) : \int_{\C} |f(z)|^2 e^{-L|z|^2} dm(z)<+\infty\}
\]
and consider the Gaussian entire function
\[
 f_L(z)=\sum_{k=0}^\infty a_k e_k(z)\ ,
\]
where $a_k$ are i.i.d. complex standard Gaussians and $\{e_k(z)\}_{k=0}^\infty$ is an orthonormal basis of $\mathcal F_L$. The Edelman-Kostlan formula gives $\mathbb E[Z_{f_L}]=\dfrac{L}{\pi}\, dm(z)$, and for a test function $\varphi$,
\[
 I_L(\varphi)=L\int_{\mathbb C} \varphi(z)\; \frac{dm(z)}{\pi}=\int_{\mathbb C} \varphi(w/\sqrt L)\; \frac{dm(z)}{\pi}\ .
\]
In particular
\[
 \mathbb E[\#(Z_{f_L}\cap D(0,r))]= \mathbb E[\#(Z_{f_1}\cap D(0,r\sqrt L))],
\]
and therefore studying the asymptotics as $L\to \infty$ is equivalent to replacing $L$ by $r^2$ and letting $r\to\infty$.

Sodin and Tsirelson proved \cite{ST3}*{Theorem 1} that, as $r\to\infty$,
\[
 e^{-Cr^4}\leq \mathbb P[ Z_{f_1}\cap D(0,r)=\emptyset]\leq e^{-cr^4}.
\]
Zrebiec extended this result to $\C^n$ \cite{Zr}*{Theorem 1.2}, showing that the decay rate is then $e^{-Cr^{2n+2}}$, which matches with our Theorem~\ref{holethm}.

Shiffman, Zelditch and Zrebiec proved also a hole theorem for sections of powers  of a positive Hermitian line bundle over a compact K\"ahler manifold \cite{SZZ}*{Theorem 1.4}. In that case the decay rate of the hole probability is again $e^{-C N^{n+1}}$.

A final word about notation. By $A\lesssim B$ we mean that there exists $C>0$ independent of the relevant variables of $A$ and $B$ for which $A\leq CB$. Then $A\simeq B$ means that $A\lesssim B$ and $B\lesssim A$.

\section{Linear statistics}

\begin{proof}[Proof of Theorem~\ref{LS}]
The proof is as in \cite{HKPV}*{Section 3.5}, so we keep it short. By Stokes and Fubini's theorems
 \begin{align*}
  \Var [I_L(\varphi)]&=\E\left[| I_L(\varphi)-\E(I_L(\varphi))|^2\right]=\E\left[ \left| \int_{\B_n} \varphi\wedge \frac i{2\pi}\partial\bar\partial \log\bigl(\frac{|f_L|^2}{K_L(z,z)}\bigr)\right|^2 \right]\\
  &=4\E\left[\left| \int_{\B_n}   \log\bigl(\frac{|f_L|}{\sqrt{K_L(z,z)}}\bigr) \frac i{2\pi}\partial\bar\partial\varphi \right|^2 \right]\\
  &=4 \int_{\B_n} \int_{\B_n} \E\left[ \log\bigl(\frac{|f_L(z)|}{\sqrt{K_L(z,z)}}\bigr) \log\bigl(\frac{|f_L(w)|}{\sqrt{K_L(w,w)}}\bigr)\right] \frac i{2\pi}\partial\bar\partial\varphi(z) \frac i{2\pi}\partial\bar\partial\varphi (w)\ .
 \end{align*}
 
Consider the normalised GAF 
\[
 \hat f(z)=\frac{f_L(z)}{\sqrt{K_L(z,z)}}\ .
\]
Then $(\hat f(z),\hat f(w))$ has joint gaussian distribution with mean 0 and marginal variances 0. Since $\hat f(z)\sim N_{\C}(0,1)$ the expectation $\E(\log|\hat f(z)|)$ is constant, and integrated against $\partial\bar\partial \varphi$ gives 0. Therefore, in the integral above, the expectation can be replaced by 
 \[
  \Cov (\log|\hat f(z)|,\log|\hat f(w)|)=
  \E[\log|\hat f(z)|\log|\hat f(w)|]-\E[\log|\hat f(z)|] \E[\log|\hat f(w)|]\ .
 \]
This yields the following bi-potential expression of the variance, which is our starting point:
\begin{align}\label{bipotential}
 \Var[I_L(\varphi)]&=
 \int_{\B_n}\int_{\B_n}  \rho_L(z,w) \frac i{2\pi}\partial\bar\partial\varphi(z) \frac i{2\pi}\partial\bar\partial\varphi (w)\\
 &=\int_{\B_n}\int_{\B_n}  \rho_L(z,w) D\varphi(z) D\varphi(w) d\mu(z) d\mu(w) \ , \nonumber
\end{align}
where $\rho_L(z,w)=4\Cov (\log|\hat f(z)|,\log|\hat f(w)|)$. By \cite{HKPV}*{Lemma 3.5.2}
\begin{align*}
  \rho_L(z,w)=\sum_{m=1}^\infty\frac{|\theta_L(z,w)|^{2m}}{m^2}\ ,
 \end{align*}
 where
\begin{equation}\label{norm-kernel}
\theta_L(z,w)=\frac{K_L(z,w)}{\sqrt{K_L(z,z)} \sqrt{K_L(w,w)}}=\frac{(1-|z|^2)^{L/2}(1-|w|^2)^{L/2}}{(1-\bar z\cdot w)^L}
\end{equation}
is the normalised covariance kernel of $f_L$.

We see next that only the near diagonal part of the double integral \eqref{bipotential} is relevant. 
Let $\varepsilon_L=2/L^{n+1}$, and split the integral in three parts
\begin{align}
  \Var[I_L(\varphi)]&=\int_{\rho_L(z,w)\leq \varepsilon_L} \rho_L(z,w) D\varphi(z) D\varphi(w) d\mu(z) d\mu(w) \tag{I1} \\
  &\quad + \int_{\rho_L(z,w)> \varepsilon_L} \rho_L(z,w) (D\varphi(z)-D\varphi(w)) D\varphi(w) d\mu(z) d\mu(w) \tag{I2}\\
  &\quad + \int_{\rho_L(z,w)> \varepsilon_L} \rho_L(z,w) ( D\varphi(w))^2 d\mu(z) d\mu(w)\tag{I3}\ .
\end{align}

The bound for the first integral is straight-forward,
\[
 |\textrm{I1}|\leq \varepsilon_L \int_{\rho_L(z,w)\leq \varepsilon_L} |D\varphi(z) D\varphi(w)| d\mu(z) d\mu(w)\leq \varepsilon_L\left(\int_{\B_n} |D\varphi(z)|\, d\mu(z)\right)^2\ .
\]

In order to bound (I2) let $\phi_z$ denote the automorphism of $\B_n$ exchanging $z$ and 0, so that $|\theta_L(z,w)|^2=(1-|\phi_z(w)|^2)^L$ (see \eqref{dh}.
By the uniform continuity of $i\partial\bar\partial\varphi$ there exists $\eta(t)$ with $\lim\limits_{t\to 1} \eta(t)=0$ such that for all $z,w\in\B_n$,
\[
 |D\varphi(z)-D\varphi(w)|\leq \eta(1-|\phi_z(w)|^2)\ .
\]
An immediate estimate shows that
\[
 x\leq\sum_{m=1}^\infty\frac{x^m}{m^2}\leq 2x\qquad x\in[0,1]\ ,
\]
and therefore
\begin{equation}\label{est-rho}
 (1-|\phi_z(w)|^2)^L\leq \rho_L(z,w)\leq 2 (1-|\phi_z(w)|^2)^L\ .
\end{equation}
By the invariance by automorphisms of the measure $d\mu$, we get (after changing appropriately the value of $C_\varphi$ at each step)
\begin{align*}
 |\textrm{I2}|&\leq 2 C_\varphi\int\limits_{\{\rho_L(z,w)>\varepsilon_L\}\cap(supp\; \varphi\times supp\; \varphi)} (1-|\phi_z(w)|^2)^L\, \eta(1-|\phi_z(w)|^2) d\mu(z) d\mu(w)\\
 &\leq C_\varphi\; \eta((\varepsilon_L/2)^{1/L})\int\limits_{\{\rho_L(z,w)>\varepsilon_L\}\cap(supp\; \varphi\times supp\; \varphi)} (1-|\phi_z(w)|^2)^L d\mu(z) d\mu(w)\\
&\leq C_\varphi\; \eta((\varepsilon_L/2)^{1/L})\int_{supp\;\varphi}\left(\int_{z: \rho_L(z,0)>\varepsilon_L} (1-|z|^2)^L d\mu(z)\right) d\mu(w)\\
&\leq C_\varphi\; \eta((\varepsilon_L/2)^{1/L}) \int_{z: \rho_L(z,0)>\varepsilon_L} (1-|z|^2)^L d\mu(z)\ .
\end{align*}

Since $\eta(t)\lesssim |1-t|$ for $t$ near 1, we see that
\[
 \eta((\varepsilon_L/2)^{1/L})\lesssim 1-(\varepsilon_L/2)^{1/L}\simeq\frac{\log L}L
\]
and therefore
\[
 |\textrm{I2}|\lesssim \frac{\log L}L \int_{z: \rho_L(z,0)>\varepsilon_L} (1-|z|^2)^L d\mu(z)\ .
\]

On the other hand, using again the invariance, we see that
\begin{align*}
 \textrm{I3}&=
 \left(\int_{\B_n} (D\varphi(w))^2 d\mu(w)\right) \int_{z: \rho_L(z,0)>\varepsilon_L} (1-|z|^2)^L d\mu(z)\ .
\end{align*}
Since $\lim\limits_{L\to\infty} \varepsilon_L^{1/L}=1$ we have thus $\textrm{I2}=\textrm{o}(\textrm{I3})$ and therefore
\begin{equation}\label{ls}
 \Var[I_L(\varphi)]=\textrm{I3} \bigl(1+\textrm{O}(\frac{\log L}L)\bigr) +\textrm{O}(\varepsilon_L)\ .
\end{equation}

It remains to compute the second factor in I3:
\[
 J:=\int_{z: \rho_L(z,0)>\varepsilon_L} \rho_L(z,0) d\mu(z)=\sum_{m=1}^\infty\frac 1{m^2} \int_{z: \rho_L(z,0)>\varepsilon_L} (1-|z|^2)^{mL} d\mu(z)\ .
\]

By \eqref{est-rho},
\[
 \{|z|^2<1-\varepsilon_L^{1/L}\}\subset \{\rho_L(z,0)>\varepsilon_L\}\subset \{|z|^2<1-(\varepsilon_L/2)^{1/L}\}
\]
and therefore
\[
 J=\sum_{m=1}^\infty\frac 1{m^2} \int\limits_{|z|^2<1-(\frac{\varepsilon_L}2)^{1/L}} (1-|z|^2)^{mL} d\mu(z)- \int\limits_{
 \stackrel{|z|^2<1-(\frac{\varepsilon_L}2)^{1/L}}{\rho_L(z,0)\leq \varepsilon_L}
 } (1-|z|^2)^{mL} d\mu(z)\ .
\]

\textit{Claim 1:} The sum of the negative terms is negligible. More precisely,
\[
\sum_{m=1}^\infty\frac 1{m^2} \int\limits_{
 \stackrel{|z|^2<1-(\frac{\varepsilon_L}2)^{1/L}}{\rho_L(z,0)\leq \varepsilon_L}
 } (1-|z|^2)^{mL} d\mu(z) = \textrm{O} \left(\frac{\log^{n-1} L}{L^{2n+1}}\right)\ .
\]

Assuming this we have
\begin{equation}\label{J}
 J= \sum_{m=1}^\infty\frac 1{m^2} I_m +\textrm{o}(L^{-n})
\end{equation}
where, denoting $r_L=1-(\frac{\varepsilon_L}2)^{1/L}$,
\[
 I_m=\int_{|z|^2<r_L} (1-|z|^2)^{mL} d\mu(z)=\int_{|z|^2<r_L} (1-|z|^2)^{mL-n-1} d\nu(z)\ .
\]
Integration in polar coordinates (\cite{Ru}*{1.4.3}) shows that $I_m$ is a truncated beta function:
\begin{align*}
 I_m&= n \int_0^{\sqrt{r_L}}(1-r^2)^{mL-n-1}r^{2(n-1)} 2r\; dr
 =n \int_0^{r_L} (1-t)^{mL-n-1} t^{n-1} dt \ .
\end{align*}

A repeated integration by parts yields, for $n,k>0$,
 \begin{multline*}
n \int_0^{r} (1-t)^{k-1} t^{n-1} dt=\\
=\frac{n! \Gamma(k)}{\Gamma(n+k)} \bigl(1-(1-r)^{k+n-1}\bigr)-\sum_{j=1}^{n-1}\frac{n! \Gamma(k)}{\Gamma(n-j)\Gamma(k+j)} (1-r)^{k+j-1} r^{n-j}\ ,
 \end{multline*}
thus taking $k=mL-n$ we deduce from \eqref{J} that
\begin{multline*}
J=n!\sum_{m=1}^\infty \frac 1{m^2}\left[\frac{\Gamma(mL-n)}{\Gamma(mL)}[1-(1-r_L)^{mL-1}]-\right. \\
\left. \sum_{j=1}^{n-1}\frac{\Gamma(mL-n)}{\Gamma(n-j)\Gamma(mL-n+j)} (1-r_L)^{mL-n+j-1} r_L^{n-j}\right]\ .
\end{multline*}

\textit{Claim 2:} The negative terms in this sum are again negligible. Specifically, 
\[
 \sum_{m=1}^\infty \frac 1{m^2} \sum_{j=1}^{n-1}\frac{\Gamma(mL-n)}{\Gamma(n-j)\Gamma(mL-n+j)} (1-r_L)^{mL-n+j-1} r_L^{n-j}= \textrm{O}\left(\frac{\log^{n+j} L}{L^{2n+3}}\right).
\]

The asymptotics of the $\Gamma$-function
\begin{equation}\label{asymptoticsGamma}
 \lim_{m\to\infty}\frac{\Gamma(m+n)}{\Gamma(m) m^n}=1 
\end{equation}
and the fact that $(1-r_L)^{mL}=(\varepsilon_L/2)^m$ tends to 0 as $L\to \infty$ yield
\begin{align*}
 J&=n!\sum_{m=1}^\infty \frac 1{m^2}\frac{\Gamma(mL-n)}{\Gamma(mL)}+ \textrm{o}(L^{-n})
 =n!\sum_{m=1}^\infty \frac 1{m^2}\frac{1}{(mL)^n}+ \textrm{o}(L^{-n})\\
 &=n! \frac 1{L^n}\zeta(n+2)+ \textrm{o}(L^{-n}).
\end{align*}
Plugging this in \eqref{ls} we finally obtain the stated result.
\end{proof}

\textit{Proof of Claim 1}.
Denote by $N$ the sum we need to estimate. Using that $\varepsilon_L=2L^{-(n+1)}$, unwinding the condition $\rho_L(z,0)\leq \varepsilon_L$ a rough estimate yields
\begin{align*}
 N&=\sum_{m=1}^\infty \frac 1{m^2}\int\limits_{(\frac{\varepsilon_L}2)^{1/L}\leq 1-|z|^2\leq \varepsilon_L^{1/L}} (1-|z|^2)^{mL-n-1} d\nu(z)\\
 &\lesssim\sum_{m=1}^\infty \frac 1{m^2} (\varepsilon_L^{1/L})^{L-n-1}\; \nu\bigl(\{1-\varepsilon_L^{1/L}\leq |z|^2\leq 1-(\frac{\varepsilon_L}2)^{1/L}\}\bigr)\\
 &\lesssim \varepsilon_L^{1-\frac{n}L} \left(1-\frac 1{2^{1/L}}\right) \left(1-(\frac{\varepsilon_L}2)^{1/L}\right)^{n-1}\\
 &\leq \frac 2{L^{n+1}}\left(\frac{\log 2}L+\textrm{o}(L^{-1})\right)\left(\frac{n+1}L\log L + \textrm{o}(\frac{\log L}L)\right)^{n-1}=
 \textrm{O} \left(\frac{\log^{n-1} L}{L^{2n+1}}\right)\ .
\end{align*}

\textit{Proof of Claim 2}.
We have
\begin{align*}
(1-r_L)^{mL-n+j-1} r_L^{n-j}&=L^{-\frac{n+1}L (mL-n+j-1)}\left(\frac{n+1}L \log L+ \textrm{o}(L^{-n})\right)^{n-j}\\
&=\textrm{O}\left(\frac{\log^{n+j} L}{L^{(n+1)m+n+j}}\right) \ .
\end{align*}
On the other hand, the number of terms in the sum in $j$ is independent of $L$, so by \eqref{asymptoticsGamma},
for $L$ big enough and for all $j$
\[
 \lim_{L\to\infty} \frac{\Gamma(mL-n)}{\Gamma(mL-n+j)}=\frac 1{(mL)^j}\ .
\]
Thus, denoting by $M$ the double sum in $m$ and $j$ we see that
\begin{align*}
 M&\simeq \sum_{m=1}^{\infty} \frac 1{m^2} \sum_{j=1}^{n-1} \frac 1{(mL)^j}\frac{\log^{n+j} L}{L^{(n+1)m+n+j}}=
 \textrm{O}\left(\frac{\log^{n+j} L}{L^{2n+3}}\right)\ .
\end{align*}

As an immediate consequence of the results of M. Sodin and B. Tsirelson and the previous computations we obtain the asymptotic normality of $I_L(\varphi)$.

\begin{corollary}\label{normality} As $L\to\infty$ the distribution of the normalised variables
 \[
\frac{I_L(\varphi)-\E[I_L(\varphi)]}{\sqrt{\Var(I_L(\varphi))}}
\]
tend weakly to the standard (real) gaussian.
\end{corollary}

\begin{proof}
Consider the normalised GAF $\hat f_L(z)$, whose covariance kernel is $\theta_L(z,w)$. Notice that
\begin{align*}
 J_L(\varphi)&:=\int_{\B_n}\log |\hat f_L(z)|^2 D\varphi(z)\; d\mu(z)=I_L(\varphi)-\int_{\B_n}\log K_L(z,z)\; D\varphi(z)\; d\mu(z)\ ,
\end{align*}
and that the second term has no random part. Hence $(J_L(\varphi)-\E[J_L(\varphi)])/\sqrt{\Var[J_L(\varphi)]}$ and $(I_L(\varphi)-\E[I_L(\varphi)])/\sqrt{\Var[I_L(\varphi)]}$ have the same distribution, and 
according to \cite{ST1}*{Theorem 2.2}, to prove the asymptotic normality of $I_L(\varphi)$ it is enough to see that
\begin{align*}
 \textrm{(a)}\qquad & \liminf_{L\to\infty} \frac{\int_{\B_n} \int_{\B_n} |\theta_L(z,w)|^2 D\varphi(z)\, D\varphi(w)\, d\mu(z)\; d\mu(w)}
 {\sup\limits_{w\in\B_n} \int_{\B_n} |\theta_L(z,w)|\; d\mu(z)} >0 \\
  \textrm{(b)}\qquad & \lim_{L\to\infty} \sup_{w\in\B_n} \int_{\B_n} |\theta_L(z,w)|\; d\mu(z)=0\ .
\end{align*}

By the invariance under automorphisms of the measure $\mu$
\[
 \int_{\B_n} |\theta_L(z,w)|\; d\mu(z)=\int_{\B_n} (1-|z|^2)^{L/2}\; d\mu(z)\ ,
\]
and (b) follows. 

On the other hand the double integral in the numerator of (a) is essentially the same we have found in the proof of the previous theorem (see \eqref{bipotential}), and the same computations show that (a) holds. 
\end{proof}

\section{Large deviations}

We begin with the proof of Corollary~\ref{conseqSLD} (assuming Theorem~\ref{smoothlargedeviations}).

\begin{proof}[Proof of Corollary~\ref{conseqSLD}]
 Since $\omega^{n-1}\wedge [Z_{f_L}]$ is a positive current, the functional $I_L(\psi)$ is monotone, i.e., if $\psi_1\leq \psi_2$ then $I_L(\psi_1)\leq I_L(\psi_2)$.
 
 Let $\psi_1,\psi_2$ be smooth compactly supported functions in $\B_n$ such that $0\leq \psi_1\leq \chi_U\leq \psi_2\leq 1$ and
 \[
  \int_{\B_n} \psi_1\; d\mu \geq \mu(U)(1-\delta)\ ,\qquad \int_{\B_n}  \psi_2\; d\mu\leq \mu(U)(1+\delta).
 \]
Outside an exceptional set of probability $e^{-cL^{n+1}}$ we have, by Theorem~\ref{smoothlargedeviations},
 \begin{align*}
  I_L(U)&\leq I_L(\psi_2)\leq (1+\delta) \E[I_L(\psi_2)]=(1+\delta)L\int_{\B_n} \psi_2 d\mu\leq (1+\delta)^2 L\mu(U)\ .
 \end{align*}
Similarly, using $\psi_1$, we see that 
\[
 I_L(U)\geq (1-\delta)^2 L\mu(U)
\]
outside another set of probability $e^{-cL^{n+1}}$, which after appropiately changing the value of $\delta$ completes the proof.
\end{proof}

A different proof of Corollary~\ref{conseqSLD} can be obtained by following the scheme of \cite{HKPV}*{Theorem 7.2.5}, using the Poisson-Szeg\"o representation of the averages $\int_{|\xi|=1}\log|f_L(\xi)|d\sigma(\xi)$ instead of Jensen's formula.

\textit{Proof of Theorem \ref{smoothlargedeviations}}. 
Applying Stokes' theorem, we have
\begin{align*}
I_{L}(\varphi)-\E\left[I_{L}(\varphi)\right]& 
 =\int_{\B_n}\varphi\wedge \frac{i}{2\pi}\partial\overline{\partial}\log\frac{|f_L|^2}{K_L(z,z)}
=\int_{\B_n}\log\frac{|\hat f_L|^2}{K_L(z,z)}\frac{i}{2\pi}\partial\overline{\partial}\varphi.
\end{align*}
Thus,  
\[
|I_{L}(\varphi)-\E[I_{L}(\varphi)]|\leq \|D\varphi\|_\infty \int_{\text{supp}\varphi}\left|\log |\hat f_{L}(z)|^2\right|d\mu(z).
\]
By \eqref{E}, the proof of Theorem \ref{smoothlargedeviations} will be completed as soon as we prove the following Lemma.
\begin{lemma}\label{mainlemmaSLD}
For any regular compact set $K$ and any $\delta>0$ there exists $c=c(\delta,K)$ such that
\[
\mathbb P\left[ \int_{K}\left|\log |\hat f_L(z)|^2 \right|d\mu(z)>\delta L \right]\leq e^{-cL^{n+1}}.
\]
\end{lemma}
The key ingredient in the proof of this lemma is given by the following control on the average of 
$\bigl|\log |\hat f_L|^2\bigr|$ over pseudo-hyperbolic balls. 

\begin{lemma}\label{controlmean}
There exists a constant $c>0$ such that for a hyperbolic ball $E=E(z_0,s)$, $z_0\in\B_n$, $s\in (0,1)$, 
\[
\mathbb P\left[\frac{1}{\mu(E)}\int_{E}\left|\log |\hat f_L(\xi)|^2 \right|d\mu(\xi)>5L\mu(E)^{1/n} \right]\leq e^{-cL^{n+1}}.
\] 
\end{lemma}

Let us see first how this allows to complete the proof of Lemma~\ref{mainlemmaSLD}, and therefore of Theorem~\ref{smoothlargedeviations}. 

\begin{proof}[Proof of Lemma~\ref{mainlemmaSLD}] Cover $K$ with pseudohyperbolic balls $E_j=E(\lambda_j,\eps)$, $j=1,\dots,N$ of fixed invariant volume $\mu(E_j)=\eta$ (to be determined later on). A direct estimate shows that $N\simeq \mu(K)/\eta$. 

By Lemma~\ref{controlmean}, outside an exceptional event of probability $Ne^{-cL^{n+1}}\leq e^{-c'L^{n+1}}$,
\begin{align*}
\int_{K}\left|\log|\hat f_L(\xi)|^2\right|d\mu(\xi)&\leq \sum^{N}_{j=1}\int_{E_j}\left|\log |\hat f_L(\xi)|^2 \right|d\mu(\xi) \leq \sum^{N}_{j=1} 5L\eta^{1+1/n}
 \simeq L\mu(K) \eta^{1/n}.
\end{align*}
Choosing $\eta$ such that $\mu(K)\eta^{1/n}=\delta$ we are done.
\end{proof}

Now we proceed to prove Lemma~\ref{controlmean}. A first step is the following lemma.

\begin{lemma}\label{controlmax}
Fix $r<1$ and $\delta>0$. There exists $c>0$ and $L_0=L_0(r,\delta)$ such that for all $L\geq L_0$ and all $z_0\in\B_n$
\begin{itemize}
 \item[(a)] $P \bigl[\max\limits_{E(z_0,r)}\log|\hat f_L(z)|^2<-\delta L \bigr]\leq e^{-cL^{n+1}}$, 
 \item[(b)] $P \bigl[\max\limits_{E(z_0,r)}\log|\hat f_L(z)|^2>\delta L \bigr]\leq e^{-ce^{L\delta/2}}$.
\end{itemize}
Combining both estimates $\mathbb P\bigl[\max\limits_{E(z_0,r)}\left|\log |\hat f_L(z)|^2\right|>\delta L\bigr]\leq e^{-cL^{n+1}}$.

\end{lemma}

\begin{proof} By the invariance of the distribution of $\hat f$, it is enough to consider the case $z_0=0$.

(a) Consider the event
 \[
  \mathcal E_1=\left\{ \max_{|z|\leq r}\log|\hat f_L(z)|^2<-\delta L\right\}\ .
 \]
Note that
\begin{align*}
\log |\hat f_L(z)|^2=\log\frac{|f_L(z)|^2}{K_L(z,z)}& =\log|f_L(z)|^2-\log \frac{1}{(1-|z|^2)^{L}},
\end{align*}
hence, by subharmonicity,
\begin{align*}
\mathcal E_1& \subset \left\{\max_{|z|\leq r} \log |f_L(z)|^2 \leq L\log\frac{1}{1-r^2}-L \delta\right\}\\
&=\left\{ \max_{|z|=r} \log |f_L(z)|^2 \leq L\bigl(\log\frac{1}{1-r^2}- \delta\bigr)  \right\}.
\end{align*}
 
Therefore, letting $\tilde\delta=\frac{\delta}2[\log(\frac 1{1-r^2})]^{-1}$,
\[
\mathbb P[\mathcal E_1]\leq \mathbb P\left[ \max_{|z|=r}\frac{\log|f_L(z)|}{L}\leq \bigl(\frac{1}{2}-\tilde{\delta}\bigr)\log\frac{1}{1-r^2} \right]\ .
\]
The estimate of $\mathbb P[\mathcal E_1]$ will be done as soon as we prove the following lemma, which is the analogue of the upper bound in \cite{HKPV}*{Lemma 7.2.7}.

\begin{lemma}\label{maxlogfL}
For $0<\delta<1/2$ and $r\in(0,1)$ there exist $c=c(\delta,r)$ and $L_0=L_0(\delta,r)$ such that for all $L\geq L_0$
\[
\mathbb P\left[\max_{|z|=r}\frac{\log|f_L(z)|}{L}\leq \bigl(\frac{1}{2}-\delta\bigr)\log\frac{1}{1-r^2} \right]\leq e^{-cL^{n+1}}\ ·
\]
\end{lemma}

\begin{proof} [Proof of Lemma~\ref{maxlogfL}]
Under the event we want to estimate
\[
\max_{|z|=r}|f_L(z)|\leq (1-r^2)^{-L\left(\frac{1}{2}-\delta\right)}\ .
\]
We shall see that this implies that some coefficients of the series of $f_L$ are necessarily ``small", something that only happens with a probability less than $e^{-cL^{n+1}}$. Since
\[
f_L(z)=\sum_{\alpha}\frac{\partial^{\alpha}f_L(0)}{\alpha!}z^{\alpha}=\sum_{\alpha}a_{\alpha}\left(\frac{\Gamma(|\alpha|+L)}{\alpha!\Gamma(L)}\right)^{1/2}z^{\alpha},
\]
we have
\[
a_{\alpha}=\left(\frac{\alpha!\Gamma(L)}{\Gamma(L+|\alpha|)}\right)^{1/2}\frac{\partial^{\alpha}f_L(0)}{\alpha!}, 
\]
and by Cauchy's formula \cite{Ru}*{pag.37}
\[
\frac{\partial^{\alpha}f_L(0)}{\alpha!}=\frac{\Gamma(n+|\alpha|)}{\Gamma(n)\alpha!r^{|\alpha|}}\int_{S}f_L(r\xi)\overline{\xi}^{\alpha}d\sigma(\xi).
\]
Hence
\[
|a_{\alpha}|\leq \left(\frac{\alpha!\Gamma(L)}{\Gamma(L+|\alpha|)}\right)^{1/2}\frac{\Gamma(n+|\alpha|)}{\Gamma(n)\alpha!r^{|\alpha|}} \left(\max_{\xi\in S}|\xi^{\alpha}|\right) \left(\max_{|z|=r}|f_L|\right)
\]

Since for $m\in\mathbb N$,
\begin{equation}\label{maxmonomial}
 \sum_{|\alpha|=m}\frac{\alpha^{\alpha}}{\alpha!|\alpha|^{|\alpha|}}=\frac{1}{m!},
\end{equation}
we have
\begin{align*}
|a_{\alpha}|&  
 \leq \left(\frac{\Gamma(L)}{\Gamma(L+|\alpha|)}\right)^{1/2}
\frac{\Gamma(n+|\alpha|)}{\Gamma(n)}
\left(\frac{\alpha^{\alpha}}{\alpha!|\alpha|^{|\alpha|}}\right)^{1/2}
(1-r^{2})^{-L\left(\frac{1}{2}-\delta\right)} r^{-|\alpha|}.
\end{align*}

Using 
\begin{equation}\label{sumlevels}
\sum_{|\alpha|=m}\frac{\alpha^{\alpha}}{\alpha!|\alpha|^{|\alpha|}}=\frac{1}{m!},
\end{equation}
Stirling's formula and the asymptotics for the Gamma function \eqref{asymptoticsGamma}, we get (for $m\gg n$)
\begin{align*}
\sum_{|\alpha|=m}|a_{\alpha}|^2 &\leq \frac{\Gamma(L)}{\Gamma(L+m)}\frac{\Gamma^2(n+m)}{\Gamma^2(n) m!} r^{-2m}(1-r^2)^{-L(1-2\delta)}\\
&\lesssim \frac{\Gamma(L)\Gamma(n+m)}{\Gamma(L+m)} m^{n-1} r^{-2m}(1-r^2)^{-L(1-2\delta)} \\
& \lesssim \frac{L^L (m+n)^{m+n}}{(L+m)^{L+m}} m^{n-1} r^{-2m}(1-r^2)^{-L(1-2\delta)}\\
&\lesssim \frac{L^L (m+n)^{m}}{(L+m)^{L+m}} m^{2n} r^{-2m}(1-r^2)^{-L(1-2\delta)}
\end{align*}

(We use this lemma (and Lemma~\ref{controlmax}) in the proof of Lemma~\ref{controlmean}, which is in turn used in Lemma~\ref{mainlemmaSLD} with a radius $r=\epsilon$ such that $\mu(E(\lambda_j,\epsilon))=(\delta/\mu(K))^n$. Since in Lemma~\ref{mainlemmaSLD} it is enough to consider $\delta$ small, here it is enough to consider $r$ close to 0. We assume thus that $r$ is close to 0, although the proof seems to work for all $r\in (0,1)$).

For the indices $m$ such that
\begin{equation}\label{c1}
m\leq  \frac{r^2L-n}{1-r^2} 
\end{equation}
we have $(1-r^2)m\leq r^2L-n$ and therefore $\frac{m+n}{L+m} r^{-2}\leq 1$. Hence
\[
 \sum_{|\alpha|=m}|a_{\alpha}|^2 \leq \frac{L^L }{(L+m)^{L}}  \frac{m^{2n}}{(1-r^2)^{L(1-2\delta)}}=
 \left[\frac{L m^{\frac {2n}L}}{(L+m) (1-r^2)^{1-2\delta}}\right]^L\ .
\]

Fix $\epsilon$ (possibly very small) and let us find conditions on $m$ so that the term in the brackets is smaller than $(1+\epsilon)^{-1}$. Assume that $m$ satisfies \eqref{c1} and
\begin{equation}\label{c2}
 m\geq (1-\delta) \frac{r^2 L}{1-r^2} ,
\end{equation}
Then $\lim_{L\to\infty} m^{\frac {2n}L}=1$ and we can take $L_0$ such that $m^{\frac {2n}L}\leq 1+\epsilon$ for $L\geq L_0$. Then, for the term in the brackets to be smaller than $(1+\epsilon)^{-1}$ it is enough to have
\[
 \frac{ L (1+\epsilon)}{(L+m) (1-r^2)^{1-2\delta}}\leq \frac 1{1+\epsilon}\ ,
\]
that is
\[
 (1+\epsilon)^2 L\leq (L+m) (1-r^2)^{1-\delta}\ .
\]
This will occur for the $m$'s in our range if 
\[
 (1+\epsilon)^2<\bigl(1+\frac{(1-\delta)r^2}{1-r^2}\bigr) (1-r^2)^{1-\delta}\ .
\]
Thus for the existence of an $\epsilon>0$ with this property it is enough to have
\[
 1<\bigl(1+\frac{(1-\delta)r^2}{1-r^2}\bigr) (1-r^2)^{1-\delta}= (1-r^2)^{1-\delta}+\frac{(1-\delta)r^2}{(1-r^2)^{\delta}}\ .
\]
The function $f(x)=(1-x)^{1-\delta}+\frac{(1-\delta)x}{(1-x)^{\delta}}$ has $f(0)=1$ and $f'(x)=\frac{\delta(1-\delta) x}{(1-x)^{1+\delta}}>0$, thus $f(x)>1$ for $x>0$.

All combined, for the indices $m$ satisfying \eqref{c1} and \eqref{c2}, i.e. in the set
\[
A:=\bigl\{m: \ (1-\delta)\frac{r^2 L}{1-r^2}\leq m\leq \frac{r^2 L-n}{1-r^2}\bigr\}
\]
the following estimate holds
\[
 \sum_{|\alpha|=m}|a_{\alpha}|^2 \lesssim (1+\epsilon)^{-m}
\]
Let us see next that this happens with very small probability. Note that
\begin{align*}
 \mathbb P\left[\sum_{|\alpha|=m} |a_{\alpha}|^2\leq (1+\eps)^{-m},\ \forall m\in A\right]
&=\prod_{m\in A} \mathbb P\left[\sum^{N(n,m)}_{j=1}|\xi_{j}|^2\leq (1+\eps)^{-m}\right] ,
\end{align*}
where $\xi_j\sim \mathbb N_{\C}(0,1)$ are independent and $N(n,m)=\Gamma(n+m)/(m!\Gamma(n))$ is the number of indices $\alpha$ with $|\alpha|=m$. 
The variable $\sum^{N(n,m)}_{j=1}|\xi_j|^2$ follows a Gamma distribution of parameter $N(n,m)$, 
therefore,
\begin{align*}
\mathbb P\left[\sum^{N(n,m)}_{j=1}|\xi_{j}|^2\leq (1+\eps)^{-m}\right]&
=\frac{1}{\Gamma(N(n,m))}\int^{(1+\eps)^{-m}}_{0}x^{N(n,m)-1}e^{-x}dx\\
&\leq \frac{1}{\Gamma(N(n,m))}\frac{1}{N(n,m)}(1+\eps)^{-mN(n,m)}.
\end{align*}

Observe that for $m\in A$, $m \simeq L$ and, by \eqref{asymptoticsGamma}, $N(n,m)\simeq m^{n-1}\simeq L^{n-1}$. With this and Stirling's formula we get
\begin{multline*}
 \log \mathbb P\left[\sum^{N(n,m)}_{j=1}|\xi_{j}|^2\leq (1+\eps)^{-(m+n)}\right]
\lesssim -\log\Gamma(L^{n-1})-\log L^{n-1}-L\cdot L^{n-1}\log(1+\eps)\\
 \simeq -L^{n}\log(1+\eps)\left[1+o(1)\right]\leq -CL^{n}.
\end{multline*}
Therefore, changing appropiately the value $C$ at each step, we finally see that
\[
\mathbb P\left[\sum_{|\alpha|=m}|a_{\alpha}|^2\leq (1+\eps)^{-m},\ \forall m\in A\right]
\leq \left(e^{-CL^n}\right)^{\#A}=\left(e^{-CL^n}\right)^{L+o(1)}\leq e^{-CL^{n+1}}.
\]
This finishes the proof of (a) in Lemma~\ref{controlmax}.
\end{proof}

 (b) Let now 
 \[
  \mathcal E_2:=\left\{\max_{|z|\leq r}\log|\hat f_L(z)|^2>\delta L\right\}=\left\{\max_{|z|\leq r}\left[\log |f_L(z)|-\frac{L}{2}\log\frac{1}{1-|z|^2}\right]>\delta L\right\}.
 \]
We estimate the probability of this event by controlling the coefficients of the series of $f_L$.
Let $C$ be a constant to be determined later on. Split the sum defining $|f_L|$ as
\begin{align}
|f_L(z)|&\leq  \sum_{|\alpha|\leq C\delta L}|a_{\alpha}|\left(\frac{\Gamma(|\alpha|+L)}{\alpha!\Gamma(L)}\right)^{1/2}|z^{\alpha}|+\sum_{|\alpha|>C\delta L}|a_{\alpha}|\left(\frac{\Gamma(|\alpha|+L)}{\alpha!\Gamma(L)}\right)^{1/2}|z^{\alpha}|\\
&=:(I)+(II).\nonumber
\end{align}
We shall estimate each part separately. 

Let us begin with the first sum. Using Cauchy-Schwarz inequality,  \eqref{maxmonomial} and \eqref{sumlevels} we obtain
\begin{align*}
(I) & \leq \left(\sum_{|\alpha|\leq C\delta L}|a_{\alpha}|^2\right)^{1/2}
\left(\sum_{|\alpha|\leq C\delta L}\frac{\Gamma(|\alpha|+L)}{\alpha!\Gamma(L)}\frac{\alpha^{\alpha}}{|\alpha|^{|\alpha|}}|z|^{2|\alpha|}\right)^{1/2}\\
& =\left(\sum_{|\alpha|\leq C\delta L}|a_{\alpha}|^2\right)^{1/2}\left(\sum_{m\leq C\delta L}\frac{\Gamma(m+L)}{m!\Gamma(L)}|z|^{2m}\right)^{1/2}\\
& \leq \left(\sum_{|\alpha|\leq C\delta L}|a_{\alpha}|^2\right)^{1/2}(1-|z|^2)^{-L/2}=\left(\sum_{|\alpha|\leq C\delta L}|a_{\alpha}|^2\right)^{1/2}\sqrt{K_L(z,z)}.
\end{align*}

Now we shall see that, except for an event of small probability, $(II)$ is bounded (if $C$ is choosen appropiately). For $|z|\leq r$,
\begin{align*}
(II)& \leq \sum_{|\alpha|>C\delta L}|a_{\alpha}|\left(\frac{\Gamma(|\alpha|+L)}{\alpha!\Gamma(L)}\right)^{1/2}\left(\frac{\alpha^{\alpha}}{|\alpha|^{|\alpha|}}\right)^{1/2}r^{|\alpha|}
& \leq \sum_{|\alpha|>C\delta L}|a_{\alpha}|\left(\frac{\Gamma(|\alpha|+L)}{|\alpha|!\Gamma(L)}\right)^{1/2}r^{|\alpha|}
\end{align*}
Let $\beta>0$ be such that $r=e^{-\beta}$ and consider $\gamma\in (0,\beta)$ and $\eps>0$ such that $0<\gamma<\gamma+\eps<\beta$. Define the following event:
\[
A=\left\{ |a_{\alpha}|\leq e^{\gamma|\alpha|},\ \forall\alpha : |\alpha|\geq C\delta L\right\}.
\]
If $A$ occurs, by the asymptotics \eqref{asymptoticsGamma}, 
\begin{align*}
(II)&\leq \sum_{m>C\delta L} e^{\gamma m} \left(\frac{\Gamma(m+L)}{m!\Gamma(L)}\right)^{1/2}r^{m} \frac{\Gamma(m+n)}{\Gamma(n) m!}\\ 
&\lesssim \frac{1}{\sqrt{\Gamma(L)}}\sum_{m>C\delta L}m^{\frac{L-1}{2}}m^{n-1}e^{\gamma m}r^{m}
\leq \frac{1}{\sqrt{\Gamma(L)}}\sum_{m>C\delta L}m^{n+L/2} (e^{\gamma}r)^{m}\ .
\end{align*}

\begin{lemma}
 Given $\epsilon>0$ there exists $C>0$ big enough so that for all $m>C\delta L$
\[
\frac{m^{n+L/2}}{\sqrt{\Gamma(L)}}\leq C e^{\eps m}.
\]
\end{lemma}
\begin{proof}
It is enough to see that there exists a constant $D$ such that for $x>C\delta L$
\[
 f(x):=\epsilon x -(n+\frac L2)\log x+\frac 12\log  \Gamma(L)+ D\geq 0 \ .
\]
Note that $\lim_{x\to\infty}f(x)=+\infty$ and that $f$ is increasing for $x\geq\epsilon^{-1} (n+L/2)$. Choose $C$ with $C\delta L>\epsilon^{-1} (n+L/2)$, so that $f$ is increasing for $x>C\delta L$. Then, by Stirling's formula,
\begin{align*}
 f(C\delta L)&=\epsilon C\delta L-(n+\frac{L}2)\log(C\delta L)+\frac 12 \log\Gamma(L)+\log D\\
 &=\epsilon C\delta L -(n+\frac{L}2)\log(C\delta)-n\log L +\frac 12 \log(\frac{2\pi}L)^{1/2}-\frac{L}2+\textrm{O}(1)\\
 &=[\epsilon C\delta- \frac 12 \log(C\delta)-\frac 12]L + \textrm{o}(L)\ .
\end{align*}
Choose $C$ big enough so that the term in the brackets is positive, and therefore $f(x)>0$ for $x>C\delta L$.
\end{proof}

Taking $C$ as in this lemma we obtain
\[ 
(II)\lesssim \sum_{m>C\delta L}e^{-[\beta-(\gamma+\eps)]m}\leq \frac{1}{1-e^{-[\beta-(\gamma+\eps)]}}\ .
\]

Now we show that the event $A$ has ``big'' probability. The variables $|a_\alpha|^2$ are independent exponentials, hence
\[
\mathbb P[A]=\prod_{|\alpha|\geq C\delta L}1-\mathbb P[|a_{\alpha}|\geq e^{\gamma|\alpha|}]
=\prod_{m\geq C\delta L}\left[1-e^{-e^{2\gamma m}}\right]^{\frac{\Gamma(n+m)}{\Gamma(n)m!}}.
\]
Since $x=e^{-e^{2\gamma m}}$ is close to 0, we can use the estimate $\log(1-x)\simeq -x$. Thus, using \eqref{asymptoticsGamma} once more,
\[
\log \mathbb P[A]=\sum_{m\geq C\delta L}\frac{\Gamma(n+m)}{\Gamma(n)m!}\log\left[1-e^{-e^{2\gamma m}}\right]\simeq -\sum_{m\geq C\delta L}m^{n-1}e^{-e^{2\gamma m}}.
\]
There exists $L_0$ such that for all $L\geq L_0$ and $m\geq C\delta L$,
\[
 m^{n-1} e^{-e^{2\gamma m}}\leq e^{-e^{\gamma m}},
\]
and therefore
\[
\log \mathbb P[A]\geq -\sum_{m\geq C\delta L}e^{-e^{\gamma m}}\simeq -e^{-e^{\gamma C\delta L}} .
\]
Choosing $C$ big enough so that, in addition to the previous conditions, $\gamma C>\log\frac{1}{1-r^2}$  we have
\[
-e^{-e^{(2\gamma-\eta)C\delta L}} > -e^{-(1-r^2)^{-\delta L}}
\]
and therefore
\[
 \mathbb P[A]\geq e^{-e^{-(1-r^2)^{-\delta L}}}\ .
\]

So far we have proved that, after choosing $\gamma$ appropriately, and under the event $A$:
\[
|f_L(z)|\leq \left(\sum_{|\alpha|\leq C\delta L}|a_{\alpha}|^2\right)^{1/2}\sqrt{K_L(z,z)}+C_r.
\]
Therefore, the condition
\[
\frac{|f_L(z)|^2}{K_L(z,z)}>e^{\delta L}
\]
imposed in $\mathcal E_2$ implies that, for $|z|\leq r$ and $L$ big,
\[
\sum_{|\alpha|\leq C\delta L}|a_{\alpha}|^2\geq \left(e^{\frac{\delta}{2}L}-\frac{C_r}{\sqrt{K_L(z,z)}}\right)^2>\frac 12 e^{ \delta L}.
\]
Let
\[
M_L=\#\left\{\alpha : \ |\alpha|\leq C\delta L\right\}=\sum_{m\leq C\delta L}\frac{\Gamma(n+m)}{\Gamma(n)m!}\leq C\delta L\frac{\Gamma(n+C\delta L)}{\Gamma(n)(C\delta L)!}\simeq C^n\delta^nL^n.
\]
Hence,
\begin{align*}
\mathbb P[A\cap \mathcal E_2]&\leq \mathbb P\left[\bigl\{\sum_{|\alpha|\leq C\delta L} |a_{\alpha}|^2 \geq \frac 12 e^{ \delta L} \bigl\}\right]
\leq \sum_{|\alpha|\leq C\delta L} \mathbb P\left[ |a_{\alpha}|^2\geq\frac{e^{\delta L}}{2 M_L}\right]\\
&= M_L e^{-(\frac{e^{\delta L}}{2 M_L})}\leq e^{-e^{\frac{\delta}2 L}}\ .
\end{align*}

Using this last estimate and the bound for $\mathbb P[A]$, we have finally that
\[
\mathbb P[\mathcal E_2]\leq e^{-e^{L\delta/2}} .
\]
\end{proof}

It remains to prove Lemma \ref{controlmean}. Before we proceed we need the following mean-value estimate of $\log |\hat f_L(\lambda)|^2$.

\begin{lemma}\label{invlemma}
Let $\lambda\in\B_n$, $s>0$ and consider the pseudo-hyperbolic ball $E(\lambda,s)$. Then
\[
\log |\hat f_L(\lambda)|^2 \leq \frac{1}{\mu (E(\lambda,s))}\int_{E(\lambda,s)}\log |\hat f_L(\xi)|^2 d\mu(\xi)+L\eps(n,s),
\]
where
\[
\eps(n,s)=\frac{n}{\mu(E(0,s))}\int^{\frac{s^2}{1-s^2}}_{0}x^{n-1}\log(1+x)dx\leq \frac{s^2}{1-s^2}=\mu(E(\lambda,s))^{1/n}.
\]
\end{lemma}

\begin{proof}
 By the subharmonicity of $\log|f_L(z)|^2$ we have
 \begin{align*}
 \log |\hat f_L(\lambda)|^2 &\leq \frac{1}{\mu (E(\lambda,s))}\int_{E(\lambda,s)}\log | f_L(\xi)|^2 d\mu(\xi) +L\log(1-|z|^2)\\
 &=\frac{1}{\mu (E(\lambda,s))}\int_{E(\lambda,s)}\log |\hat f_L(\xi)|^2 d\mu(\xi)+\\
 &\qquad \qquad + L\left[\log(1-|\lambda|^2) - \frac{1}{\mu (E(\lambda,s))}\int_{E(\lambda,s)}\log  (1-|\xi|^2) d\mu(\xi)\right]\ .
 \end{align*}
Identity \eqref{dh} and the pluriharmonicity of $\log|1-\bar\lambda\cdot \xi|^2$ yield
 \begin{align*}
  \frac{1}{\mu (E(\lambda,s))}\int_{E(\lambda,s)}\log  (1-|\xi|^2) d\mu(\xi)&=
  \frac{1}{\mu (B(0,s))}\int_{B(0,s)}\log  (1-|\phi_\lambda(\xi)|^2) d\mu(\xi)\\
  &= \log(1-|\lambda|^2)+\frac{1}{\mu (B(0,s))}\int_{B(0,s)}\log (1-|\xi|^2) d\mu(\xi).
 \end{align*}
Changing into polar coordinates and performing the change of variable $x=\frac{r^2}{1-r^2}$ we get
\[
 \int_{B(0,s)}\log (1-|\xi|^2) d\mu(\xi)=2n\int_0^{s}\log(1-r^2) \frac{r^{2n-1}}{(1-r^2)^{n+1}}\; dr=-n\int_0^{\frac{s^2}{1-s^2}} x^{n-1} \log(1+x)\; dx.
\]
This and the fact that $\mu(B(0,s))=\frac{s^{2n}}{(1-s^2)^n}$ (\cite{Stoll} (4.4)) finish the proof.
\end{proof}

\begin{proof}[Proof of Lemma \ref{controlmean}]
According to Lemma~\ref{controlmax}(a), except for an exceptional event of probability $e^{-cL^{n+1}}$, there is $\lambda\in E:=E(z_0,s)$ such that
\[
-L(\mu(E))^{1/n}<\log  |\hat f_L(\lambda)|^2.
\]
Therefore, using Lemma \ref{invlemma}, 
\[
-L(\mu(E))^{1/n}<\frac{1}{\mu(E)}\int_{E}\log  |\hat f_L(\xi)|^2 d\mu(\xi)+L(\mu(E))^{1/n}.
\]
Hence
\[
0 < \frac{1}{\mu(E)}\int_{E}\log |\hat f_L(\xi)|^2 d\mu(\xi) +2L(\mu(E))^{1/n}.
\]
Separating the positive and negative parts of the logarithm we obtain:
\[
\frac{1}{\mu(E)}\int_{E}\log^{-}  |\hat f_L(\xi)|^2 d\mu(\xi)
\leq \frac{1}{\mu(E)}\int_{E}\log^{+} |\hat f_L(\xi)|^2 d\mu(\xi)+2 L(\mu(E))^{1/n}.
\]
Hence,
\[
\frac{1}{\mu(E)}\int_{E}\left|\log  |\hat f_L(\xi)|^2 \right|d\mu(\xi)\leq \frac{2}{\mu(E)}\int_{E}\log^{+}  |\hat f_L(\xi)|^2 d\mu(\xi)+2 L(\mu(E))^{1/n}.
\]
Again by Lemma \ref{controlmax}, outside another exceptional event of probability $e^{-cL^{n+1}}$,
\[
\frac{1}{\mu(E)}\int_{E}\left|\log  |\hat f_L(\xi)|^2 \right|d\mu(\xi)\leq
2\max_{E}\log^{+}  |\hat f_L(\xi)|^2 +2 L(\mu(E))^{1/n}
\leq 5 L\mu(E)^{1/n}.
\]
\end{proof}

\section{The hole theorem}

Here we prove Theorem~\ref{holethm}.

The upper bound is a direct consequence of the results in the previous section. Letting $U=B(0,r)$ and applying Corollary \ref{conseqSLD} with $\delta\mu(U)$ instead of $\delta$ we get 
\[
 \mathbb P\left[Z_{f_L}\cap B(0,r)=\emptyset\right]\leq\mathbb P\left[|I_L(U)-L\mu(U)|>\delta L\mu(U)\right]\leq e^{-C_{2}L^{n+1}}.
\]

The method to prove the lower bound is by now standard (see for example \cite{HKPV}*{Theorem 7.2.3} and \cite{ST1}): we shall choose three events forcing $f_L$ to have a hole $B(0,r)$ and then we shall see that the probability of such events is at least $e^{-C_{1}L^{n+1}}$. Our starting point is the estimate
\[
 |f_L(z)|\geq |a_0|-\left|\sum_{0<|\alpha|\leq CL} a_\alpha\left(\frac{\Gamma(L+|\alpha|)}{\alpha! \Gamma(L)}\right)^{1/2} z^\alpha\right|-
 \left|\sum_{|\alpha|> CL} a_\alpha\left(\frac{\Gamma(L+|\alpha|)}{\alpha! \Gamma(L)}\right)^{1/2} z^\alpha\right|\ ,
\]
where $C$ will be choosen later on.

The first event is
\[
 E_1:=\left\{\ |a_0|\geq 1\right\}\ ,
\]
which has probability
\[
\mathbb P[E_1]=\mathbb P[|a_0|^2\geq 1]=e^{-1}.
\]
The second event corresponds to the tail of the power series of $f_L$. Let
\[
E_2:=\left\{\ |a_\alpha|\leq \sqrt{\frac{\alpha!\Gamma(n)}{\Gamma(n+|\alpha|)}}|\alpha|^n,\quad 
\forall \alpha: \ |\alpha|>CL\right\}.
\]

We shall see next that $\mathbb P[E_2]$ is big, and that under the event $E_2$ the tail of the power series of $f_L$ is small.

Using \eqref{maxmonomial} we have:
\begin{align*}
\left|\sum_{|\alpha|>CL}a_{\alpha}\left(\frac{\Gamma(L+|\alpha|)}{\alpha!\Gamma(L)}\right)^{1/2}z^{\alpha}\right|&\leq 
\sum_{|\alpha|>CL}|a_{\alpha}|\left[\frac{\Gamma(L+|\alpha|)}{\Gamma(L)\alpha!}
\frac{\alpha^{\alpha}}{|\alpha|^{|\alpha|}}r^{2|\alpha|}\right]^{1/2}\\
&\leq \sum_{m>CL}\left[\frac{\Gamma(L+m)}{\Gamma(L)}r^{2m}\right]^{1/2}
\sum_{|\alpha|=m}|a_\alpha|\left(\frac{\alpha^{\alpha}}{\alpha!|\alpha|^{|\alpha|}}\right)^{1/2} .
\end{align*}
Thus, using Cauchy-Schwarz inequality and \eqref{sumlevels}:
\begin{align*}
\left|\sum_{|\alpha|>CL}a_{\alpha}\left(\frac{\Gamma(L+|\alpha|)}{\alpha!\Gamma(L)}\right)^{1/2}z^{\alpha}\right| & \leq 
 \sum_{m>CL}\left[\frac{\Gamma(L+m)}{\Gamma(L)m!}r^{2m}\right]^{1/2}
\left(\sum_{|\alpha|=m}|a_{\alpha}|^2\right)^{1/2} .
\end{align*}
Using the asymptotics of the Gamma function \eqref{asymptoticsGamma}, we estimate
\[
\frac{\Gamma(m+L)}{\Gamma(L) m!}\simeq \frac{m^{L-1}}{\Gamma(L)}\leq 
\left[\frac{m^{L/m}}{\Gamma(L)^{1/m}}\right]^{m}.
\]
Note that the function $g(x):=\left(x^{L}/\Gamma(L)\right)^{1/x}$ is decreasing for $x\geq L$. 
Thus if $m>CL$ Stirling's formula yields
\[
\frac{m^{L/m}}{\Gamma(L)^{1/m}}\leq \frac{(CL)^{1/C}}{\Gamma(L)^{1/(CL)}}=\frac{C^{1/C} L^{1/(2CL)} e^{1/C}}{(2\pi)^{1/(2CL)}} [1+\textrm{o}(1)]
\leq (eC)^{\frac 1C} K^{\frac 1{2C}},
\]
where $K=\max\limits_{x>0} x^{1/x}=e^{-1/e}$.

Let $h(C)=(eC)^{\frac 1C} K^{\frac 1{2C}}$ and note that $h(C)>1$ and $\lim\limits_{C\to \infty} h(C)=1$. Hence, there exists $C$ big enough so that 
$h(C)r^2\leq (1-\delta)^2$ and therefore,
\begin{align*}
\left|\sum_{|\alpha|>CL}a_{\alpha}\left(\frac{\Gamma(L+|\alpha|)}{\alpha!\Gamma(L)}\right)^{1/2}z^{\alpha}\right| &\leq 
\sum_{m>CL}\left[h(C)r^2\right]^{m/2}\left(\sum_{|\alpha|=m}|a_{\alpha}|^2\right)^{1/2}\\
& \leq \sum_{m>CL}(1-\delta)^m\left(\sum_{|\alpha|=m}|a_{\alpha}|^2\right)^{1/2}.
\end{align*}
Under the event $E_2$,
\[
\sum_{|\alpha|=m}|a_{\alpha}|^2\leq 
\sum_{|\alpha|=m}\frac{|\alpha|!\Gamma(n)}{\Gamma(n+|\alpha|)}|\alpha|^{2n}=m^{2n},
\]
hence the tail of $f_L$ is controlled by the tail of a convergent series and there exists $C$ big enough so that:
\[
\left|\sum_{|\alpha|>CL}a_{\alpha}\left(\frac{\Gamma(L+|\alpha|)}{\alpha!\Gamma(L)}\right)^{1/2}z^{\alpha}\right|\leq \sum_{m>CL}(1-\delta)^m m^{2n}<\frac 14.
\]

Now we prove that the probability of $E_2$ is big. Since the variables $a_{\alpha}$ are 
independent we have, again by \eqref{asymptoticsGamma}:
\begin{align*}
\mathbb P[E_2^c] & \leq 
\sum_{|\alpha|>CL}\mathbb P\left[|a_{\alpha}|^2>\frac{|\alpha|!\Gamma(n)}{\Gamma(n+|\alpha|)}|\alpha|^{2n}\right]
=\sum_{m>CL}\mathbb P\left[|\xi|^2>\frac{m!\Gamma(n)}{\Gamma(n+m)}m^{2n}\right]\frac{\Gamma(n+m)}{\Gamma(n)m!}\\
& \lesssim \sum_{m>CL}\mathbb P\left[|\xi|^2>c_nm^{n+1}\right]m^{n-1}=\sum_{m>CL}e^{-c_nm^{n+1}}m^{n-1}.
\end{align*}
Thus for $L$ big enough, $\mathbb P[E_2^c]\leq 1/2$, and $\mathbb P[E_2]\geq 1/2$.

The third event takes care of the middle terms in the power series of $f_L$. Let
\[
E_3:=\left\{\ |a_\alpha|^2<\frac{1}{16CL}\frac{|\alpha|!\Gamma(n)}{\Gamma(n+|\alpha|)}(1-r^2)^{L}\quad \forall \alpha:\ 0<|\alpha|\leq CL\right\}.
\]
Using Cauchy-Schwarz's inequality, \eqref{maxmonomial} and\eqref{sumlevels} we get, as in previous computations: 
\begin{align*}
& \left|\sum_{0<|\alpha|\leq CL}a_{\alpha}\left(\frac{\Gamma(L+|\alpha|)}{\alpha!\Gamma(L)}\right)^{1/2}z^{\alpha}\right|\leq 
\left(\sum_{0<|\alpha|\leq CL}|a_{\alpha}|^2\right)^{1/2}
\left(\sum_{0<|\alpha|\leq CL}\frac{\Gamma(|\alpha|+L)}{\Gamma(L)\alpha!}\frac{\alpha^{\alpha}}{|\alpha|^{|\alpha|}}r^{2|\alpha|}\right)^{1/2}\\
& \leq \left(\sum_{0<|\alpha|\leq CL}|a_{\alpha}|^2\right)^{1/2}\left(\sum_{0<m\leq CL}\frac{\Gamma(m+L)}{\Gamma(L)m!}r^{2m}\right)^{1/2}
\leq \left(\sum_{0<|\alpha|\leq CL}|a_{\alpha}|^2\right)^{1/2}(1-r^2)^{-L/2}.
\end{align*}
Under the event $E_3$,
\[
\sum_{0<|\alpha|\leq CL}|a_{\alpha}|^2\leq \sum_{0<m\leq CL}\frac{1}{16CL}(1-r^2)^{L}
=\frac{1}{16}(1-r^2)^{L},
\]
and therefore
\[
\left|\sum_{0<|\alpha|\leq CL}a_{\alpha}\left(\frac{\Gamma(L+|\alpha|)}{\alpha!\Gamma(L)}\right)^{1/2}z^{\alpha}\right|\leq \frac{1}{4}.
\]
On the other hand,
\[
\mathbb P[E_3]=\prod_{0<m\leq CL}\left[1-e^{-\frac{1}{16CL}\frac{m!\Gamma(n)}{\Gamma(m+n)}(1-r^2)^{L}}\right]^{\frac{\Gamma(n+m)}{m!\Gamma(n)}}
\]
Note that if $L$ is big enough then the term appearing in the exponential is small. Since $1-e^{-x}\geq x/2$ for $x\in(0,1/2)$, we get
\begin{align*}
\mathbb P[E_3] &\geq \prod_{0<m\leq CL}\left[\frac{1}{32CL}\frac{m!\Gamma(n)}{\Gamma(m+n)}(1-r^2)^{L}\right]^{\frac{\Gamma(n+m)}{m!\Gamma(n)}}\\
& =\left[\frac{\Gamma(n)}{32CL}(1-r^2)^{L}\right]^{\sum\limits_{0<m\leq CL}\frac{\Gamma(n+m)}{m!\Gamma(n)}}\prod_{0<m\leq CL}\left(\frac{m!}{\Gamma(m+n)}\right)^{\frac{\Gamma(n+m)}{m!\Gamma(n)}}.
\end{align*}
Now we estimate each term of the product and the sum by the ``worst" term. Denote $M=[CL]$. The exponent in the first factor is controlled by
\begin{align*}
\sum^{M}_{m=1}\frac{\Gamma(n+m)}{m!\Gamma(n)}\leq M\frac{\Gamma(n+M)}{M!\Gamma(n)}=\frac{\Gamma(n+M)}{\Gamma(M)\Gamma(n)}
\leq M^{n}\leq (CL)^n .
\end{align*}
Similarly, for the second factor we have
\begin{align*}
\prod^{M}_{m=1}\left(\frac{m!}{\Gamma(m+n)}\right)^{\frac{\Gamma(n+m)}{m!\Gamma(n)}}
&\geq \left(\frac{M!}{\Gamma(M+n)}\right)^{M\frac{\Gamma(n+M)}{M!\Gamma(n)}}
\geq  \left(\frac{M!}{\Gamma(M+n)}\right)^{\frac{\Gamma(n+M)}{\Gamma(M)}}\\
&\geq  \left(\frac{\Gamma(CL+1)}{\Gamma(CL+n)}\right)^{\frac{\Gamma(n+CL)}{\Gamma(CL)}} .
\end{align*}
Then, using again \eqref{asymptoticsGamma},
\begin{align*}
\log\mathbb P[E_3]& \geq (CL)^n\log\left[\frac{\Gamma(n)}{32CL}(1-r^2)^{L}\right]+
\frac{\Gamma(n+CL)}{\Gamma(CL)}\log\left[\frac{\Gamma(CL+1)}{\Gamma(CL+n)}\right]\\
& \succsim (CL)^n \log\left[\frac{\Gamma(n)}{32CL}(1-r^2)^L\right]+
(CL)^n\log(CL)^{1-n}\\
&=C^nL^n \left[\log\frac{\Gamma(n)}{32C^n}-n \log L-L\log\frac{1}{1-r^2}\right]\\
& =-C^nL^{n+1}\log\frac{1}{1-r^2}\left[1+\frac{n\log L}{L\log\frac{1}{1-r^2}}-\frac{\log\frac{\Gamma(n)}{32 C^n}}{L\log\frac{1}{1-r^2}}\right]\\
& =-C^nL^{n+1}\log\frac{1}{1-r^2}\left[1+o(1)\right].
\end{align*}

Finally, 
\[
\mathbb P[E_2\cap E_3\cap\mathcal C]\geq e^{-C(n)\log\left(\frac{1}{1-r^2}\right)L^{n+1}\left[1+o(1)\right]},
\]
and under this event $|f_L(z)|\geq 1-1/4-1/4 >0$.

\end{document}